\documentclass[11pt,reqno]{amsart}
\usepackage{enumerate, latexsym, amsmath, amsfonts, amssymb, amsthm, color}
\def\pmod #1{\ ({\rm{mod}}\ #1)}
\def\Z{\mathbb Z}
\def\N{\mathbb N}

\def\l{\left}
\def\r{\right}
\def\bg{\bigg}
\def\({\bg(}
\def\){\bg)}
\def\t{\text}
\def\f{\frac}

\def\gs{\geqslant}

\def\al{\alpha}

\def\eq{\equiv}

\def\Ack{\medskip\noindent {\bf Acknowledgments}}
\newcommand{\e}{\mathrm e}
\newcommand{\1}{\mathbf 1}
\newcommand{\meas}{\operatorname{meas}}
\newcommand{\Sing}{\mathfrak S}
\newcommand{\M}{{\mathfrak M}_{+}}
\newcommand{\MD}{{\mathfrak M}_{-}}
\newcommand{\mm}{\mathfrak m}
\newcommand{\eps}{\varepsilon}
\newcommand{\SF}{\mathrm F}
\newcommand{\SL}{\mathrm L}
\newcommand{\Rea}{\operatorname{Re}}
\newcommand{\calA}{\mathcal A}
\newcommand{\calB}{\mathcal B}
\theoremstyle{plain}
\newtheorem{theorem}{Theorem}

\newtheorem{lemma}{Lemma}
\newtheorem{corollary}{Corollary}
\newtheorem{proposition}{Proposition}

\theoremstyle{definition}

\theoremstyle{remark}
\newtheorem{remark}{Remark}

 \vspace{4mm}

\begin{document}

\hbox{Preprint}
\medskip

\title
[{On sums of two primes and six Fibonacci numbers}]
{On sums of two primes \\ and six Fibonacci numbers}

\author
[Z.-W. Sun] {Zhi-Wei Sun}

\address{School of Mathematics, Nanjing
University, Nanjing 210093, People's Republic of China}
\email{zwsun@nju.edu.cn}

\keywords{Additive bases, circle method, Fibonacci numbers, Lucas numbers, primes, representations of positive integers.
\newline \indent Supported by the Natural Science Foundation of China (grant no. 12371004).}

\begin{abstract}  Via the circle method, we deduce that
all sufficiently large positive integers are sums of two primes and six
positive Fibonacci numbers, and are also sums of two primes and seven Lucas numbers.
We also prove that each sufficiently large positive integer can be written as a sum of two primes, three Fibonacci numbers, and three Lucas numbers.
\end{abstract}
\maketitle

\section{Introduction}
\setcounter{lemma}{0}
\setcounter{theorem}{0}
\setcounter{corollary}{0}
\setcounter{proposition}{0}
\setcounter{remark}{0}

In 1934, using the Brun sieve, Romanoff \cite{Roma} proved that a positive proportion
of the odd integers can be written $p+2^n$, where $p$ is a prime and $n\in\Z^+=\{1,2,3,\ldots\}$. 
In 1972 Crocker \cite{Cro} showed that there are infinitely many positive odd integers not of the form $p+2^a+2^b$, where $p$ is a prime and $a,b\in\Z^+$. 

The binary Goldbach conjecture states that any even integer greater than two can be written as a sum of two primes. Although this remains open, Linnik \cite{Linnik51,Linnik53} managed to prove that there is a positive integer $k$
such that each sufficiently large even number is a sum of two primes and $k$ positive powers of two.
It is known that we may take
$$k=54000,\ 25000,\ 2250,\ 1906,\ 13,\ 8,\ 7$$
by Liu, Liu and Wang \cite{LLW}, Li \cite{Li1}, Wang \cite{Wang}, Li \cite{Li2},
Heath-Brown and Schlage--Puchta \cite{HeathBrownPuchta02}, Pintz and Ruzsa \cite{PR20},
and Johnston and Trudgian \cite{JT26}, respectively.
Gallagher's density method \cite{Gallagher75} and the
large-values refinements of Pintz and Ruzsa \cite{PR03,PR20} form the standard framework for
these quantitative versions.  The work of
Bombieri and Davenport \cite{BD66} about the shifted-prime
sieve and its further refinements also play important roles in recent improvements
given in \cite{HeathBrownPuchta02,JT26}.

The classical Fibonacci numbers are given by
$$F_0=0,\ F_1=1, \ \t{and}\ F_{n+1}=F_n+F_{n-1}\ (n=1,2,3,\ldots).$$
The related Lucas numbers are defined by
$$L_0=2,\ L_1=1,\ \t{and}\ L_{n+1}=L_n+L_{n-1}\ (n=1,2,3,\ldots).$$
It is well known that for each $n\in\N=\{0,1,2,\ldots\}$ we have the Binet formulae
$$F_n=\f1{\sqrt5}\l(\l(\f{1+\sqrt5}2\r)^n-\l(\f{1-\sqrt5}2\r)^n\r)$$
and
$$ L_n=\l(\f{1+\sqrt5}2\r)^n+\l(\f{1-\sqrt5}2\r)^n.$$
Note that
$$\sqrt5 F_n\sim L_n\sim \l(\f{1+\sqrt5}2\r)^n$$
as $n\to+\infty$.
It is easy to see that $2\mid F_n$ if and only if $3\mid n$.
For various properties of Fibonacci numbers and Lucas numbers, one may consult the book \cite{S24}.

In 2008,  the author (cf. \cite{S10}) conjectured that any integer $n>4$ can be written as the sum of an odd prime and two positive Fibonacci numbers. In 2010, Lee \cite{Lee} proved that
the set of sums of a prime and a Fibonacci number has a positive lower (asymptotic) density.
In 2025 Wang \cite{Wang25} obtained a similar result which states that the set of sums of a prime and a Lucas number has a positive lower (asymptotic) density.

Unlike the first-order recurrent sequence $(2^n)_{n\gs0}$, both $(F_n)_{n\gs0}$ and $(L_n)_{n\gs0}$
are second-order recurrent sequences. Motivated by the Linnik-Goldbach representations
of large even numbers, in this paper we study sums of two primes and few Fibonacci or Lucas numbers.

Our main results are the following three novel theorems.

\begin{theorem}\label{thm:F6}
Every sufficiently large positive integer $N$ has a representation
\[
 N=p_1+p_2+F_{n_1}+\cdots+F_{n_6},
\]
where $p_1$ and $p_2$ are primes, and $n_1,\ldots,n_6$ are positive integers.
\end{theorem}

\begin{remark} Theorem \ref{thm:F6} indicates that the set 
$$A=\{p+F_a+F_b+F_c:\ p\ \t{is prime}\ \t{and}\ a,b,c\in\Z^+\}$$
is an asymptotic basis of order two (i.e., any sufficiently large positive integer
is a sum of two elements of $A$).
\end{remark}

\begin{theorem}\label{thm:L7}
Each sufficiently large positive integer $N$ has a representation
\[
 N=p_1+p_2+L_{m_1}+\cdots+L_{m_7},
\]
where $p_1$ and $p_2$ are primes, and $m_1,\ldots,m_7$ are positive integers. 
\end{theorem}

\begin{theorem} \label{thm:mixed33}
Every sufficiently large positive integer $N$ has a representation
\[
 N=p_1+p_2+F_{a_1}+F_{a_2}+F_{a_3}
              +L_{b_1}+L_{b_2}+L_{b_3},
\]
where $p_1$ and $p_2$ are primes, and $a_1,a_2,a_3,b_1,b_2,b_3$ are positive integers.
\end{theorem}

Our proofs of these three theorems use the Goldbach--Linnik circle-method framework.  The
prime-pair input is a uniform shifted-prime upper bound, while the recurrence
sums supply many shifts and permit large-value estimates through finite carry
automata.  The main difficulty in proving Theorem \ref{thm:F6}
is that the sixth-degree minor-arc polynomial cannot be bounded sharply by a
single-sum threshold.  Splitting the Fibonacci terms by parity resolves this
only if the resulting subsums are treated jointly; separate moment bounds
cannot be multiplied because all subsums depend on the same phase.
The proofs of Theorems \ref{thm:L7} and \ref{thm:mixed33} use the same nested-major-arc framework.

The final certified margins for the above three theorems are shown in the following table:

\[
\begin{array}{c|c|c|c}
\text{case}&\text{major lower bound}&\text{minor upper bound}&\text{margin}\\ \hline
F^6,\ N\text{ even}&1.994081560286761&1.993349561749000&0.000731998537761\\
F^6,\ N\text{ odd}&1.994140081930877&1.988380036000000&0.005760045930877\\
L^7&0.996900300000000&0.942703702774701&0.054196597225299\\
F^3L^3&0.996500340000000&0.924205210500000&0.072295129500000.
\end{array}
\]

\section{Notation, parity classes, and representation integrals}\label{sec:notation}
\setcounter{lemma}{0}
\setcounter{theorem}{0}
\setcounter{corollary}{0}
\setcounter{proposition}{0}
\setcounter{remark}{0}
\setcounter{equation}{0}

Set
\[
 \varphi=\frac{1+\sqrt5}{2},\qquad \psi=\f{1-\sqrt5}2,
 \ \ \t{and}\ \vartheta=\frac{1999}{2000}.
\]
For $j\in\Z^+$, we  define
\[
 u_j^{(\SF)}=F_{j+1}\ \ \t{and}\ \ u_j^{(\SL)}=L_j=\varphi^j+\psi^j=F_{j+1}+F_{j-1}.
\]
For large integer $N>0$,  let $J=J(N)$ be the largest integer for which
\begin{equation}\label{eq:Jdef}
 L_J\le \frac{N^\vartheta}{100}.
\end{equation}
Then 
\begin{equation}\label{eq:Jsymp}
 J=\frac{\vartheta\log N}{\log\varphi}+O(1).
\end{equation}
Let $N_1=N^{0.99}$. Define
\begin{equation}\label{eq:sums}
 S(\alpha)=\sum_{N_1<p\le N}(\log p)e(p\alpha),\ \t{and}\ 
 T_\kappa(\alpha)=\sum_{j=1}^Je(u_j^{(\kappa)}\alpha) \ \t{for}\ 
 \kappa\in\{\SF,\SL\},
\end{equation}
where $e(t)=e^{2\pi it}$.  Write
$T_\kappa^\sharp(\alpha)=T_\kappa(\alpha+1/2)$. Then
\begin{equation}\label{eq:Qdef}
 \mathcal Q_\kappa(\alpha)=\frac12\bigl(|T_\kappa(\alpha)|^2+
 |T_\kappa^\sharp(\alpha)|^2\bigr).
\end{equation}
 contains exactly the pairs of
recurrence terms having the same parity.

For Theorem \ref{thm:F6}, we define
\[
 \calA=\{1\le j\le J:j\equiv2\pmod3\},\ \ 
 \calB=\{1\le j\le J:j\not\equiv2\pmod3\},
\]
\begin{equation}\label{eq:ABsums}
 A(\alpha)=\sum_{j\in\calA}e(F_{j+1}\alpha)\ \ \t{and}\ \
 B(\alpha)=\sum_{j\in\calB}e(F_{j+1}\alpha).
\end{equation}
Since $F_n$ is even exactly when $3\mid n$, every term in $A$ is even and
every term in $B$ is odd.  Moreover,
\begin{equation}\label{eq:JAB}
 J_A:=|\calA|=\frac J3+O(1)\ \ \t{and}\ \ 
 J_B:=|\calB|=\frac{2J}{3}+O(1).
\end{equation}
For even and odd $N$, respectively, put
\begin{align}
 R_{\mathrm e}(N)&=\int_0^1S(\alpha)^2A(\alpha)^2B(\alpha)^4
 e(-N\alpha)\,d\alpha,\label{eq:Re}\\
 R_{\mathrm o}(N)&=\int_0^1S(\alpha)^2A(\alpha)B(\alpha)^5
 e(-N\alpha)\,d\alpha.\label{eq:Ro}
\end{align}
Their discrete expansions have nonnegative coefficients.  Since every prime
in $S$ is odd, $R_{\mathrm e}$ counts representations with two even and four
odd Fibonacci terms, while $R_{\mathrm o}$ counts representations with one
even and five odd Fibonacci terms.  Positivity of the appropriate integral
therefore proves Theorem~\ref{thm:F6}.

For Theorems \ref{thm:L7} and \ref{thm:mixed33}, we use the exact parity projectors
\begin{align}
 \mathcal P_{\SL,7,N}(\alpha)
 &=\frac12\left(T_{\SL}(\alpha)^7+(-1)^N
 T_{\SL}^\sharp(\alpha)^7\right),\label{eq:PL7}\\
 \mathcal P_{3,3,N}(\alpha)
 &=\frac12\left(T_{\SF}(\alpha)^3T_{\SL}(\alpha)^3+(-1)^N
 T_{\SF}^\sharp(\alpha)^3T_{\SL}^\sharp(\alpha)^3\right).
 \label{eq:PM33}
\end{align}
The corresponding weighted representation integrals are
\begin{equation}\label{eq:repintegral}
 \int_0^1S(\alpha)^2\mathcal P(\alpha)e(-N\alpha)\,d\alpha.
\end{equation}
Again every surviving discrete coefficient is nonnegative and has exactly
the parity required for a sum of two odd primes.

\section{Published prime inputs and nested major arcs}\label{sec:majorarcs}
\setcounter{lemma}{0}
\setcounter{theorem}{0}
\setcounter{corollary}{0}
\setcounter{proposition}{0}
\setcounter{remark}{0}
\setcounter{equation}{0}

\subsection{The shifted-prime coefficient}

For a statement $S$, we define $\1_{S}$ as $1$ or $0$ according as $S$ holds or not.
For a nonzero integer $h$, we set 
\begin{equation}\label{eq:Sings}
 \Sing(h)=2C_2\1_{2\mid h}\prod_{\substack{p\mid h\\p>2}}
 \frac{p-1}{p-2}\ \ \ \t{with}\ C_2=\prod_{p>2}\left(1-\frac1{(p-1)^2}\right).
\end{equation}
Johnston and Trudgian \cite[Proposition~2]{JT26} proved that uniformly for $0<h<N$ we have
\[
 \#\{p_1,p_2\le N:p_1-p_2=h\}
 \le(6.7814+\eps)\frac{C_2N}{(\log N)^2}
 \prod_{\substack{p\mid h\\p>2}}\frac{p-1}{p-2}.
\]
  Multiplication by $\log p_1\log p_2\le(\log N)^2$ and the
normalization in \eqref{eq:Sings} yields that for every nonzero even integer $h$ with $|h|<N$ we have
\begin{equation}\label{eq:weightedpair}
 Z_N(h):=\sum_{\substack{N_1<p,q\le N\\p-q=h}}
 (\log p)(\log q)\le(3.3907+\eps)\Sing(h)N.
\end{equation}
We reserve $0.0003$ and use
\begin{equation}\label{eq:beta}
 \beta=3.391.
\end{equation}

\subsection{The two major-arc scales}

For the major-arc parameter $P>1$, set
\[
 \mathfrak M(P)=\bigcup_{q\le P\atop (a,q)=1}
 \left\{\alpha\pmod1:\left|\alpha-\frac aq\right|
 \le\frac{P}{qN}\right\}.
\]
Pintz \cite[Theorem~1]{Pintz23} permits the choice
$P_0=N^{0.41}$ and $\eps_0=0.005$, and hence a cutoff
\begin{equation}\label{eq:Pplus}
 N^{0.405}\le P_+\le N^{0.41}.
\end{equation}
Pintz and Ruzsa \cite[Theorem~2]{PR20} permits
$P_0=N^{2/5}$ and $\eps_0=0.001$, and hence
\begin{equation}\label{eq:Pminus}
 N^{0.399}\le P_-\le N^{0.4}.
\end{equation}
Define
\begin{equation}\label{eq:arcsets}
 \M=\mathfrak M(P_+),\ \ \MD=\mathfrak M(P_-),\ \ \t{and}\ 
 \mm=[0,1]\setminus\M.
\end{equation}
The numerical ranges give the pointwise inclusion $\MD\subset\M$.

For integers $m$ and $h$, we set
\[
 G_N(m)=\int_{\M} S(\alpha)^2e(-m\alpha)\,d\alpha\ \
 \t{and}\ \ 
 D_N(h)=\int_{\MD}|S(\alpha)|^2e(-h\alpha)\,d\alpha.
\]

\begin{proposition}[Nested major-arc package]\label{prop:nested}
For every fixed $\omega>0$, the cutoffs can be chosen so that there are a
constant $B_\omega$, a fixed integer $C_\omega$, and $O_\omega(1)$ odd
squarefree integers $q_\nu=q_\nu(N)$ with $\min_\nu q_\nu\to\infty$, for
which the following $(a)$ and $(b)$ hold uniformly.

{\rm (a)} If $m$ is even with $N/2\le m\le N$, then
\[
 |G_N(m)|\le B_\omega\Sing(m)N,
\]
and, provided no $q_\nu$ divides $C_\omega m$ we have
\begin{equation}\label{eq:Gmain}
 |G_N(m)-\Sing(m)m|\le\omega\Sing(m)N.
\end{equation}

{\rm (b)} If $h$ is even with $0<|h|\le N^\vartheta/50$, then
\[
 |D_N(h)|\le B_\omega\Sing(h)N,
\]
and, provided no $q_\nu$ divides $C_\omega h$ we have
\begin{equation}\label{eq:Dmain}
 |D_N(h)-\Sing(h)N|\le\omega\Sing(h)N.
\end{equation}
Also,
\begin{equation}\label{eq:Dzero}
 D_N(0)=(1+o(1))N\log P_-\ge(0.39-o(1))N\log N.
\end{equation}
\end{proposition}

\begin{proof}
For the Goldbach formula, use \cite[Theorem~1]{Pintz23}.  The pole--pole
term is $\Sing(m)(m-2N_1+O(1))$.  The supplementary singular series are
bounded by \cite[Lemma~1 and (1.21)]{Pintz23}.  Outside divisibility by a fixed multiple of
the finitely many generalized exceptional conductors, these supplementary
terms are $O(\omega\Sing(m)N)$.  Since $N_1=o(N)$, this gives
\eqref{eq:Gmain}; \cite[(7.7)]{Pintz23} also gives the uniform upper bound.

For differences, use Theorem~2 and Corollary~1 of \cite{PR20}; negative $h$ follows by symmetry.  
Taking odd squarefree
parts and the union of the finitely many conductors gives the common family
$q_\nu$.  The diagonal identity \eqref{eq:Dzero} is exactly \cite[(4.8)]{PR20}.  All these formulas
use the same prime interval $(N_1,N]$ as \eqref{eq:sums}.
\end{proof}

The minor-arc prime estimate needed below is
\cite[Lemma 2]{PR20}.  Since $P_+\ge N^{0.405}$, it gives
\begin{equation}\label{eq:Vaughan}
 \sup_{\alpha\in\mm}|S(\alpha)|\ll N^{4/5+o(1)}.
\end{equation}

\section{Periodic collisions and mean singular series}\label{sec:means}
\setcounter{lemma}{0}
\setcounter{theorem}{0}
\setcounter{corollary}{0}
\setcounter{proposition}{0}
\setcounter{remark}{0}
\setcounter{equation}{0}

The singular series has the positive divisor expansion
\begin{equation}\label{eq:divexp}
 \Sing(h)=2C_2\1_{2\mid h}
 \sum_{\substack{2\nmid d\mid h\\d\ \mathrm{is\ squarefree}}}a(d)
 \ \ \t{with}\ \ a(d)=\prod_{p\mid d}\frac1{p-2}.
\end{equation}
The companion certificate proves
\begin{equation}\label{eq:C2bounds}
 0.6601<C_2<0.66017
\end{equation}
by an exact Euler-product truncation.

\subsection{Prime periods and occurrence bounds}

For any integer $q>1$, let $\pi(q)$
be the smallest positive integer $m$ with $F_m\eq0\pmod q$ and $F_{m+1}\eq1\pmod q$,
and let $z(q)$ be the least positive integer $n$ for
which $q\mid F_n$. For any integer $n>1$, clearly
$$q\mid F_n\iff q\ \t{divides}\ (F_n,F_{z(q)})=F_{(n,z(q))}
\iff (n,z(q))=n\iff z(q)\mid n.$$
For any prime $p$, it is well known that $p\mid F_{p-(\f p5)}$
(cf. \cite{Wall60} and \cite{SS92}) and hence $z(p)\mid p-(\f p5)$, where $(-)$ denotes the Legendre symbol.

\begin{lemma}\label{lem:rankperiod} Let $p\not=5$ be an odd prime.
For $c=F_{z(p)+1}\pmod p$, we have
\[
 \pi(p)=z(p)\operatorname{ord}_p(c)=
 \begin{cases}
 4z(p),&z(p)\text{ odd},\\
 z(p),&z(p)\text{ even and }c=1,\\
 2z(p),&z(p)\text{ even and }c=-1.
 \end{cases}
\]
In one complete period modulo $p$, every residue occurs at most four times
in the Fibonacci sequence and at most four times in the Lucas sequence.
\end{lemma}

\begin{proof}
At $n=z(p)$, the matrix
$$Q=\begin{pmatrix}1&1\\1&0\end{pmatrix}$$ satisfies $Q^n=cI$ in $\mathbb F_p$, and hence the period is
$n\operatorname{ord}_p(c)$.  Since $\det Q=-1$, one has $c^2=(-1)^n$, which
gives the three cases.

Finally fix the parity of $n$ and put $x=\varphi^n$.  The equations
$F_n=r$ and $L_n=r$ become quadratic equations in $x$, because
$\psi^n=(-1)^nx^{-1}$.  Each has at most two roots.  Within one fixed index
parity, equality of two values of $x$ forces a recurrence-state period, so
the corresponding indices are equal modulo $\pi(p)$.  There are therefore
at most two occurrences for each index parity, hence at most four in total.
\end{proof}

\subsection{Three certified mean constants}

Define
\begin{align}
 A_{\kappa,J}&=\frac1{J^2}
 \sum_{\substack{r,s\le J,\ r\ne s\\
 u_r^{(\kappa)}\equiv u_s^{(\kappa)}\ (2)}}
 \Sing(u_r^{(\kappa)}-u_s^{(\kappa)}),\qquad
 A_\kappa=\limsup_{J\to\infty}A_{\kappa,J},
 \label{eq:Akdef}\\
 A_{\calB,J}&=\frac1{J_B^2}
 \sum_{\substack{r,s\in\calB\\r\ne s}}
 \Sing(F_{r+1}-F_{s+1}),\qquad
 A_{\calB}=\limsup_{J\to\infty}A_{\calB,J}.
 \label{eq:ABdef}
\end{align}

\begin{proposition}\label{prop:means}
We have
\begin{equation}\label{eq:meanbounds}
 A_{\SF}<1.249,\ \ 
  A_{\SL}<1.281,\ \
 \qquad A_{\calB}<2.1985.
\end{equation}
\end{proposition}

\begin{proof}
We describe the full finite reduction and the tail argument.  For odd
squarefree $d$, put $T_d=\operatorname{lcm}(\pi(d),3)$.  Over a complete
length-$T_d$ block, form the residue histograms for the two value-parity
classes of $F_{j+1}$ or $L_j$.  Their squared $\ell^2$ norms, divided by
$T_d^2$, are the exact collision densities $\delta_{\SF}(d)$ and
$\delta_{\SL}(d)$.  For $\calB$, count the indices $n\not\equiv0\pmod3$ in
a complete joint period and divide the squared histogram norm by
$(2T_d/3)^2$; denote the result by $\delta_{\calB}(d)$.

For each fixed $d$, periodicity and \eqref{eq:divexp} identify the limiting
$d$-contribution.  To control all moduli before passing to the limit, group a
squarefree $d$ by its largest prime factor $p$.  Congruence modulo $d$
implies congruence modulo $p$, so the collision density for $d$ is no larger
than the corresponding prime collision density.  Lemma~\ref{lem:rankperiod}
gives
\begin{equation}\label{eq:primecollision}
 \delta_{\SF}(p),\delta_{\SL}(p)\le\frac4{\pi(p)}
 \ \ \t{and}\ \ 
 \delta_{\calB}(p)\le\frac6{\pi(p)}\le\frac6{z(p)}.
\end{equation}
The factor $6$ in the last inequality comes from repeating a Pisano period
inside $\operatorname{lcm}(\pi(p),3)$ and then normalizing by the
$2/3$-density of $\calB$.

For a finite interval of indices the occurrence bound has an additional
$O(1/J)$ endpoint term.  This term is still summable because only moduli
dividing the product of the nonzero differences can occur.  If that product
is $H_J$, then $\log H_J=O(J^3)$ and
\[
 \sum_{d\mid H_J}a(d)=\prod_{p\mid H_J}\left(1+\frac1{p-2}\right)
 \ll(\log\log(3H_J))^2\ll(\log J)^2.
\]
For completeness, the displayed product bound follows by splitting at
$y=\log(3H_J)$: the primes at most $y$ contribute $O(\log y)$ by
\eqref{eq:Widentity}, while the logarithm of the product over primes larger
than $y$ is at most
$O(\omega(H_J)/y)=O(1/\log y)$.
After division by $J$ this is $o(1)$.  This proves that the moving-cutoff
upper bounds below apply directly to the limsups in
\eqref{eq:Akdef}--\eqref{eq:ABdef}.

Let
\[
 W(p)=\prod_{3\le q<p}\left(1+\frac1{q-2}\right).
\]
For a cutoff $D$, exact integer cyclic convolution evaluates every
squarefree $d\le D$.  The omitted moduli whose largest prime is $p\le D$ are
bounded exactly by
\[
 \frac{\delta(p)}{p-2}\bigl(W(p)-W_p(D/p)\bigr),
\]
where $W_p(X)$ is the exact $a(e)$-sum over squarefree $e\le X$ with
$P^+(e)<p$.  For the same-parity constants we take $D=10000$ and obtain,
already including $2C_2<1.32034$,
\[
 1.241697034908400\quad(\SF),\qquad
 1.274207656787892\quad(\SL).
\]
For $A_{\calB}$ we take $D=390000$ and obtain
\[
 2.191813461683394.
\]

For $D<p\le10^8$, the certificate factors
$p-(5/p)$, determines $z(p)$ by fast doubling, determines $\pi(p)$ from
Lemma~\ref{lem:rankperiod}, and sums the outward-rounded bounds in
\eqref{eq:primecollision}.  The brackets before multiplication by $2C_2$
are
\[
 0.001696971409610\quad\text{for the same-parity constants}\]
 and 
 \[
 0.000107359308192\quad\text{for }A_{\calB}.
\]

It remains to bound $p>10^8$.  The identity
\begin{equation}\label{eq:Widentity}
 W(x)=\frac12\prod_{p<x}\left(1-\frac1p\right)^{-1}
 \prod_{3\le p<x}\left(1+\frac1{p(p-2)}\right)
\end{equation}
combined with Rosser and Schoenfeld \cite[Theorem~8 and Corollary~1]{RS62}, gives
$W(x)<1.36\log x$ for $x\ge10^8$.  The second product in
\eqref{eq:Widentity} is certified to be below $1.51485$, and it is easy to see that
 $e^\gamma<1.7811$.  This last inequality is not
imported: it uses
$H_n-\log n-\gamma>1/(2n)-1/(12n^2)$ at $n=1000$ and exact atanh-series
bounds for the two logarithms.  The summand comparison behind that harmonic
inequality has derivative $t^4/(6(1+t)^3)>0$.

On a dyadic block $P<p\le2P$, split at $z(p)=Z$.  Since every prime with
$z(p)=n$ divides $F_n$,
\[
 \sum_{\substack{P<p\le2P\\z(p)\le Z}}\frac1{z(p)}
 \le\frac{Z\log\varphi}{\log P}.
\]
By \cite[Corollary~1]{RS62}, we have
\[
 \pi(2P)-\pi(P)\le1.51012\frac P{\log P}.
\]
Optimizing $Z$ therefore yields
\begin{equation}\label{eq:rankblock}
 \sum_{P<p\le2P}\frac1{z(p)}
 \le\frac{2\sqrt{1.51012\,P\log\varphi}}{\log P}.
\end{equation}
Summing \eqref{eq:rankblock} over $P=2^k10^8$ gives brackets below
$0.00330$ for the same-parity constants and below $0.00494$ for
$A_{\calB}$.  Exact rational assembly, using \eqref{eq:C2bounds}, gives
\begin{align*}
 A_{\SF}&<1.248294736140<1.249,\\
 A_{\SL}&<1.280805358019<1.281,\\
 A_{\calB}&<2.198477692073<2.1985.
\end{align*}
All numerical data are obtained with the aid of computer.
\end{proof}

\section{Exceptional conductors}\label{sec:exceptional}
\setcounter{lemma}{0}
\setcounter{theorem}{0}
\setcounter{corollary}{0}
\setcounter{proposition}{0}
\setcounter{remark}{0}
\setcounter{equation}{0}

The exceptional divisibility conditions in Proposition~\ref{prop:nested}
are removed by a recurrence analogue of Romanov's argument.

\begin{lemma}[Exceptional-modulus lemma]\label{lem:exceptional}
Let $q=q(N)$ be odd and squarefree with $q\to\infty$, and let $C$ be fixed.
For any of the families $F_{j+1}$, $L_j$, or either periodic Fibonacci
subfamily $\calA,\calB$, uniformly for $|m|\le2N$,
\begin{equation}\label{eq:exceptional}
 \sum_{\substack{j\ \mathrm{in\ the\ family}\\m\ne u_j\\
 q\mid C(m-u_j)}}\Sing(m-u_j)=o(J_{\mathrm{family}}).
\end{equation}
\end{lemma}

\begin{proof}
Delete the fixed common factor of $q$ with $30C$.  For an odd squarefree
$d$, let $T(d)$ be the maximum of $\pi(p)$ over primes $p\mid d$, with
$T(1)=1$.  Lemma~\ref{lem:rankperiod} gives at most four occurrences of a
residue in a complete prime period, and therefore
\[
 \#\{j\le J:[d,q]\mid m-u_j\}
 \ll \frac{J}{T([d,q])}+1.
\]

We first show
\begin{equation}\label{eq:Tsum}
 \sum_{d\ \mathrm{odd\ squarefree}}\frac{a(d)}{T(d)}<\infty.
\end{equation}
Every prime with $\pi(p)\le x$ divides $\prod_{n\le x}F_n$, whose logarithm
is $O(x^2)$.  Hence
\[
 \sum_{T(d)\le x}a(d)
 =\prod_{\pi(p)\le x}\left(1+\frac1{p-2}\right)
 \ll(\log x)^2.
\]
Dyadic summation proves \eqref{eq:Tsum}.

Expand \eqref{eq:exceptional} by \eqref{eq:divexp}.  The part proportional
to $J$ is $o(J)$ by dominated convergence, because
$T([d,q])\to\infty$ for each fixed $d$.  For the endpoint term, only $d$
dividing
\[
 H=\prod_{j\le J,\,m\ne u_j}|m-u_j|
\]
occur.  Since $\log H=O(J\log N)=O(J^2)$,
\[
 \sum_{d\mid H}a(d)\ll(\log\log(3H))^2\ll(\log J)^2=o(J).
\]
Fixed factors removed at the beginning alter only the implied constant.
\end{proof}

\begin{corollary}\label{cor:bad}
The total singular-series weight of exceptional recurrence differences is
$o(J^2)$, and the total singular-series weight of exceptional tuples in each
of the three representation problems is $o(J^k)$, where $k$ is the number
of recurrence terms.
\end{corollary}

\begin{proof}
For differences apply Lemma~\ref{lem:exceptional} and sum over the remaining
index.  For a $k$-tuple, fix $k-1$ indices and apply the lemma to the last
one.  The uniformity in $m$ permits the summation.
\end{proof}

\section{Quadratic energies after major-arc subtraction}\label{sec:energy}
\setcounter{lemma}{0}
\setcounter{theorem}{0}
\setcounter{corollary}{0}
\setcounter{proposition}{0}
\setcounter{remark}{0}
\setcounter{equation}{0}

\subsection{The full Fibonacci and Lucas parity energies}

Expanding \eqref{eq:Qdef}, applying \eqref{eq:weightedpair} to every
nonzero shift, and using the same finite $A_{\kappa,J}$ in both terms gives
\begin{align}
 \int_0^1|S|^2\mathcal Q_\kappa\,d\alpha
 &\le\left(\beta A_{\kappa,J}+d_0+o(1)\right)NJ^2,
 \label{eq:fullenergy}\\
 \int_{\M}D|S|^2\mathcal Q_\kappa\,d\alpha
 &\ge\left((1-\omega)A_{\kappa,J}-o(1)\right)NJ^2,
 \label{eq:majorsubtract}
\end{align}
where $\omega=10^{-4}$ and
\begin{equation}\label{eq:d0}
 d_0=\frac{\log\varphi}{\vartheta}<0.481453.
\end{equation}
For the first diagonal estimate put
$\theta(x)=\sum_{p\le x}\log p$ and use
$\sum_{p\le N}(\log p)^2\le(\log N)\theta(N)$ together with Axler's Theorem~1, inequality~(1.5) on page~2 of
\cite{Axler18}, which states
$\theta(x)<x+0.15x/\log^3x$ for every $x>1$.  The lower estimate discards
the nonnegative diagonal and uses Proposition~\ref{prop:nested} on every
nonexceptional shift; Corollary~\ref{cor:bad} absorbs the others.

Since $\MD\subset\M$, subtraction and only then the limsup bounds in
Proposition~\ref{prop:means} give
\begin{equation}\label{eq:minorenergy}
 \int_{\mm}|S(\alpha)|^2\mathcal Q_\kappa(\alpha)\,d\alpha
 \le(E_\kappa+o(1))NJ^2,
\end{equation}
where
\[
 E_\kappa=[\beta-(1-\omega)]A_\kappa+d_0
\]
with
\begin{equation}\label{eq:Evalues}
 E_{\SF}<3.467937\ \ \t{and}\ \ 
 E_{\SL}<3.544453.
\end{equation}

\subsection{The odd-Fibonacci class energy}

For the sum $B$ define the analogous full and difference-major integrals.
The off-diagonal calculation uses $A_{\calB,J}$.  Its full diagonal is
\[
 J_B\sum_{N_1<p\le N}(\log p)^2
 \le\left(d_B+o(1)\right)NJ_B^2
 \ \ \t{with}\ \
 d_B=\frac{3\log\varphi}{2\vartheta}<0.72218.
\]
Unlike the preceding estimate, we retain the difference-major diagonal.
By \eqref{eq:Dzero},
\[
 J_BD_N(0)\ge(0.39-o(1))J_BN\log N
 =(0.39d_B-o(1))NJ_B^2.
\]
Consequently
\begin{equation}\label{eq:Benergy}
 \int_{\mm}|S(\alpha)|^2|B(\alpha)|^2\,d\alpha
 \le(E_{\calB}+o(1))NJ_B^2,
\end{equation}
where
\begin{align}
 E_{\calB}
 &\le[3.391-(1-10^{-4})]A_{\calB}+0.61d_B\notag\\
 &<5.697364.\label{eq:EB}
\end{align}
The positive diagonal retained here is the numerical gain that makes the
six-Fibonacci closure possible.
\section{Large values and exact carry automata}\label{sec:largevalues}
\setcounter{lemma}{0}
\setcounter{theorem}{0}
\setcounter{corollary}{0}
\setcounter{remark}{0}
\setcounter{equation}{0}

\subsection{A positive Bessel majorant}

For $s>0$ and an integer $D\ge1$, let $I_n(s)$ denote the modified Bessel
function, and put
\begin{equation}\label{eq:Besselmajorant}
 P_{s,D}(x)=w_0+2\sum_{n=1}^D I_n(s)\cos(nx)
 \ \ \t{and}\ \ 
 w_0=\e^s-2\sum_{n=1}^D I_n(s).
\end{equation}
For every parameter pair used below, the companion certificate proves
$w_0>0$ by outward rational bounds.  The Fourier expansion of
$\e^{s\cos x}$ gives
\[
 P_{s,D}(x)-\e^{s\cos x}
 =2\sum_{n>D}I_n(s)(1-\cos nx)\ge0.
\]
Thus exponential moments are bounded by weighted counts of zero relations
with digits in $[-D,D]$.

The same second-order recurrence encodes Fibonacci and Lucas relations.
Let $c_1,\ldots,c_J$ be a signed digit word, reverse it by
$d_i=c_{J+1-i}$, set $x_{-1}=x_0=0$, and define
\begin{equation}\label{eq:carryrec}
 x_i=x_{i-1}+x_{i-2}+d_i.
\end{equation}
A direct induction gives
\[
 x_i=\sum_{r=1}^i d_rF_{i-r+1}.
\]
Consequently,
\begin{equation}\label{eq:terminalidentities}
 x_J+x_{J-1}=\sum_{j=1}^Jc_jF_{j+1} \ \ \t{and}\ \ 
 x_J+2x_{J-1}=\sum_{j=1}^Jc_jL_j,
\end{equation}
where $L_j=F_j+2F_{j-1}$ for $j\ge1$.  Thus a shifted-Fibonacci
zero relation has terminal condition $x_J+x_{J-1}=0$, and a Lucas zero
relation has terminal condition $x_J+2x_{J-1}=0$.  The mixed argument below
also uses a complete Fibonacci word whose terminal condition is $x_J=0$.

\begin{lemma}[Carry confinement]\label{lem:confinement}
If $|d_i|\le D$ and the digit word satisfies one of the three terminal
conditions
\[
 x_J+x_{J-1}=0,\qquad x_J+2x_{J-1}=0,\qquad x_J=0,
\]
then all carry states $(x_i,x_{i-1})$ lie in
$[-\lceil1.895D\rceil,\lceil1.895D\rceil]^2$.
\end{lemma}

\begin{proof}
For $P_i(z)=\sum_{r=1}^i d_rz^{i-r}$, we have
\[
 x_i=\frac{\varphi P_i(\varphi)-\psi P_i(\psi)}{\sqrt5}
 \ \ \t{and}\ \ 
 x_{i-1}=\frac{P_i(\varphi)-P_i(\psi)}{\sqrt5}.
\]
First, $|P_i(\psi)|\le D/(1-|\psi|)=D\varphi^2$.  We now justify the
corresponding bound at $\varphi$.  If $m=J-i$, then
\begin{equation}\label{eq:prefixidentity}
 P_i(\varphi)=\varphi^{-m}P_J(\varphi)
 -\sum_{r=i+1}^Jd_r\varphi^{i-r}.
\end{equation}
For the shifted-Fibonacci terminal condition,
\[
 \varphi^2P_J(\varphi)=\psi^2P_J(\psi),
\]
so $|P_J(\varphi)|\le D\varphi^{-2}$ and
\[
 |P_i(\varphi)|
 \le D\left(\varphi^{-m-2}+\sum_{k=1}^m\varphi^{-k}\right)
 \le D\varphi.
\]
For the Lucas terminal condition,
\[
 (\varphi+2)P_J(\varphi)=(\psi+2)P_J(\psi)\ \ \t{and}\ \ 
  \frac{\psi+2}{\varphi+2}=\varphi^{-2}.
\]
Using the finite geometric sum in $P_J(\psi)$ gives
$|P_J(\varphi)|\le D(1-\varphi^{-J})<D$, and therefore
\[
 |P_i(\varphi)|
 \le D\bigg(\varphi^{-m}+\sum_{k=1}^m\varphi^{-k}\bigg)
 \le D\varphi.
\]
For the zero terminal condition, the first carry formula gives
\[
 \varphi P_J(\varphi)=\psi P_J(\psi).
\]
Since $|\psi|/\varphi=\varphi^{-2}$, it follows that
$|P_J(\varphi)|\le D$, and the same prefix estimate as in the Lucas case
again yields $|P_i(\varphi)|\le D\varphi$.
Consequently both displayed carry formulas are bounded in absolute value by
\[
 \frac{D(\varphi^2+\varphi)}{\sqrt5}<1.895D.
\]
\end{proof}

\subsection{A common Fibonacci--Lucas single-sum certificate}

Take $D=4$, $s=13/20$, and the $361$ states $(a,b)\in[-9,9]^2$.  The
transition
\begin{equation}\label{eq:commontransition}
 (a,b)\longmapsto(a+b+d,a),\qquad -4\le d\le4,
\end{equation}
has weight equal to the corresponding Fourier coefficient in
$P_{13/20,4}$.  This matrix is common to the Fibonacci and Lucas relations.
Lemma~\ref{lem:confinement} puts every zero-relation path for either terminal
condition inside the certified box.  Dropping the terminal condition and
summing over all terminal states can only increase the path sum.

The certificate rounds every weight upward with denominator $10^{15}$ and
verifies a positive integer vector $v$ for which
\begin{equation}\label{eq:commoncollatz}
 Mv<1.2333v
\end{equation}
coordinatewise.  The minimum cleared-denominator slack is
$$5511384341759040000.$$  Therefore
\begin{equation}\label{eq:commonrho}
 \rho(M)<1.2333.
\end{equation}
The Collatz inequality and the positive Bessel majorant imply, uniformly in
$\kappa\in\{\SF,\SL\}$ and the phase $\gamma$,
\begin{equation}\label{eq:commonmoment}
 \int_0^1\exp\!\left(\frac{13}{20}
 \Rea\bigl(\e^{-i\gamma}T_{\kappa}(\alpha)\bigr)\right)d\alpha
 \ll (1.2333)^J.
\end{equation}
The same estimate holds for $T_{\kappa}^{\sharp}$ by translation.

Put
\begin{equation}\label{eq:commonlambda}
 \lambda=0.7673
\end{equation}
and define
\[
 E_{\kappa}^{\mathrm{large}}=
 \{\al\in[0,1]:\ |T_{\kappa}(\al)|>\lambda J\ \t{or}\ 
 |T_{\kappa}(\al)^{\sharp}|>\lambda J\}\ \ \t{for}\ \kappa\in\{\SF,\SL\}.
\]
Discretize the phase circle into $2000$ points.  For every complex $z$, one
of these phases satisfies
$\Rea(\e^{-i\gamma}z)\ge|z|\cos(\pi/2000)$.  Markov's inequality,
\eqref{eq:commonmoment}, and \eqref{eq:Jsymp} give, uniformly in $\kappa$,
\begin{equation}\label{eq:commonlargemeas}
 \meas(E_{\kappa}^{\mathrm{large}})\ll N^{-\delta_*},
\end{equation}
where
\begin{equation}\label{eq:commondelta}
 \delta_*=\frac{\vartheta}{\log\varphi}
 \left(\frac{13}{20}\,0.7673\cos\frac\pi{2000}
       -\log1.2333\right)>0.6003725>\frac35.
\end{equation}
The last inequality is checked by exact rational upper and lower bounds for
the logarithms and cosine in the companion certificate.

\subsection{The period-three six-Fibonacci product certificates}

Normalize
\[
 x(\alpha)=\frac{|A(\alpha)|}{J_A}\ \ \t{and}\ \ 
 y(\alpha)=\frac{|B(\alpha)|}{J_B}.
\]
Set
\begin{equation}\label{eq:F6lambdas}
 \lambda_{\mathrm e}=\frac{1183}{2000}\ \ \t{and}\ \ 
 \lambda_{\mathrm o}=\frac{349}{1000},
\end{equation}
and define
\begin{equation}\label{eq:F6badsets}
 E_{\mathrm e}=\{\al\in\mm:\ x(\al)y(\al)>\lambda_{\mathrm e}\}\ \ \t{and}\ \ 
 E_{\mathrm o}=\{\al\in\mm:\ x(\al)y(\al)^3>\lambda_{\mathrm o}\}.
\end{equation}
The Fibonacci indices occur with period-three type pattern $B,A,B$.  For a
parameter $s$, let $M_A(s)$ or $M_B(s)$ be the $361$-state matrix with
states $(a,b)\in[-9,9]^2$, transition
$(a,b)\mapsto(a+b+d,a)$, and the degree-four Bessel weight belonging to the
chosen position type.  Endpoint changes caused by $J\bmod3$ contribute only
a fixed multiplicative constant.

For the even case, take
\[
 u_{\mathrm e}=\frac{4843}{5000}\ \ \t{and}\ \ 
 v_{\mathrm e}=\frac{4981}{10000}.
\]
The exact positive-vector certificate proves
\begin{equation}\label{eq:F6rhoeven}
 \rho\bigl(M_B(v_{\mathrm e})M_A(u_{\mathrm e})M_B(v_{\mathrm e})\bigr)
 <1.904.
\end{equation}
If $xy>\lambda_{\mathrm e}$, weighted AM--GM gives
\[
 \frac{u_{\mathrm e}}3x+\frac{2v_{\mathrm e}}3y
 \ge2\sqrt{\frac{2\lambda_{\mathrm e}u_{\mathrm e}v_{\mathrm e}}9}.
\]
After a $2000$-point phase net, the exact rational certificate verifies
\begin{equation}\label{eq:F6deltaeven}
 \frac{\vartheta}{\log\varphi}\left(
 2\cos\frac\pi{2000}
 \sqrt{\frac{2\lambda_{\mathrm e}u_{\mathrm e}v_{\mathrm e}}9}
 -\frac13\log1.904\right)>0.60026>\frac35.
\end{equation}
Hence $\meas(E_{\mathrm e})\ll N^{-3/5-\eta_{\mathrm e}}$ for some fixed
$\eta_{\mathrm e}>0$.

For the odd case, we take
\[
 u_{\mathrm o}=\frac{4931}{10000}\ \ \t{and}\ \
 v_{\mathrm o}=\frac{7311}{10000}.
\]
The corresponding exact certificate proves
\begin{equation}\label{eq:F6rhoodd}
 \rho\bigl(M_B(v_{\mathrm o})M_A(u_{\mathrm o})M_B(v_{\mathrm o})\bigr)
 <1.886.
\end{equation}
Under $xy^3>\lambda_{\mathrm o}$, we have
\[
 \frac{u_{\mathrm o}}3x+\frac{2v_{\mathrm o}}3y
 \ge4\left(\frac{8\lambda_{\mathrm o}u_{\mathrm o}v_{\mathrm o}^3}
 {2187}\right)^{1/4},
\]
and the exact rate comparison is
\begin{equation}\label{eq:F6deltaodd}
 \frac{\vartheta}{\log\varphi}\left(
 4\cos\frac\pi{2000}
 \left(\frac{8\lambda_{\mathrm o}u_{\mathrm o}v_{\mathrm o}^3}{2187}
 \right)^{1/4}-\frac13\log1.886\right)>0.6012>\frac35.
\end{equation}
Thus $\meas(E_{\mathrm o})\ll N^{-3/5-\eta_{\mathrm o}}$.

These are genuinely joint estimates.  Orthogonality produces one signed
Fibonacci relation with position-dependent Bessel weights; the matrix
product counts that relation directly.  No independence of $A$ and $B$ is
used.
\subsection{A joint Fibonacci--Lucas product estimate}

Set
\begin{equation}\label{eq:Lambda}
 \Lambda=\frac{533}{2000}=0.2665,
\end{equation}
and define
\begin{equation}
 E_{3,3}={} \left\{\al\in[0,1]:\ \frac{|T_{\SF}(\al)|}{J}
 \left(\frac{|T_{\SL}(\al)|}{J}\right)^3\ge\Lambda\ \t{or}\ 
 \frac{|T_{\SF}^{\sharp}(\al)|}{J}
 \left(\frac{|T_{\SL}^{\sharp}(\al)|}{J}\right)^3\ge\Lambda\right\}.
 \label{eq:mixedbad}
\end{equation}
We want to prove
\begin{equation}\label{eq:mixedlargemeas}
 \meas(E_{3,3})\ll N^{-3/5-\eta_0}
\end{equation}
for some fixed $\eta_0>0$.

Use degree-three majorants with
\begin{equation}\label{eq:mixedparameters}
 s_{\SF}=\frac3{20}\ \ \t{and}\ \  s_{\SL}=\frac{17}{40}.
\end{equation}
A zero-frequency term in the joint moment has digits
$c_j,d_j\in[-3,3]$ and satisfies
\begin{equation}\label{eq:mixedrelation}
 \sum_jc_jF_{j+1}+\sum_jd_jL_j=0.
\end{equation}
Put $f_j=F_{j+1}$ for $j\ge-1$, so that $f_{-1}=F_0=0$ and
$f_0=F_1=1$, and put $d_j=0$ outside $1\le j\le J$.  Since
$L_j=f_j+f_{j-2}$, the relation is exactly
\begin{equation}\label{eq:mixedeffectiveword}
 d_2f_0+\sum_{i=1}^J(c_i+d_i+d_{i+2})f_i=0.
\end{equation}
Thus the full extended Fibonacci digit word has the endpoint digit $d_2$ at
$f_0$ and the effective digits $c_i+d_i+d_{i+2}$ at $f_i$.  Every digit has
absolute value at most $9$.  After reversing the first $J$ effective digits,
the digit used at step $i$ is $c_i+d_i+d_{i-2}$.  Two Lucas digits must
therefore be retained as memory.  The endpoint digit $d_2$ is already one of
the chosen Lucas digits; it carries no additional Bessel weight.  At the final step the endpoint digit is processed without an additional Bessel weight.  Its terminal restriction will be discarded for an upper
bound, which can only enlarge the path sum.

Take states
\begin{equation}\label{eq:mixedstates}
 (a,b,r,t),\qquad a,b\in[-23,23],\quad r,t\in[-3,3].
\end{equation}
There are $47^2\cdot7^2=108241$ states.  More explicitly, write
$\widetilde c_i=c_{J+1-i}$ and $\widetilde d_i=d_{J+1-i}$, with
$\widetilde d_i=0$ outside the word.  The interior effective digit at step
$i$ is
$\widetilde c_i+\widetilde d_i+\widetilde d_{i-2}$.  If
$r=\widetilde d_{i-2}$ and $t=\widetilde d_{i-1}$, choosing
$c=\widetilde c_i$ and $d=\widetilde d_i$ gives the transition
\begin{equation}\label{eq:mixedtransition}
 (a,b,r,t)\longmapsto(a+b+c+d+r,a,t,d)
\end{equation}
with weight $w_{\SF}(c)w_{\SL}(d)$.  Thus the matrix records exactly every choice of the first $J$ reversed effective digits.  The actual paths start from the carry state
$(0,0)$ and the memory state $(0,0)$.  If its coefficients are denoted by
$e_0,e_1,\ldots,e_J$, then reversing all $J+1$ digits gives
\[
 x_{J+1}=\sum_{i=0}^Je_iF_{i+1}=0.
\]
Thus the complete word has the zero terminal condition in
Lemma~\ref{lem:confinement}.  Applying that lemma with $D=9$ places all its
carries in $[-18,18]^2$.  Hence the certified box $[-23,23]^2$
contains every state occurring during the first $J$ transitions.  Dropping the final transition and its terminal condition, and then summing over all terminal states, only enlarges the sum by a constant independent of $J$.  The exact certificate verifies upward-rounded degree-three Bessel
weights
and a positive vector $v$ satisfying
\begin{equation}\label{eq:mixedcollatz}
 M_{3,3}v<\frac{283}{250}v.
\end{equation}
The vector has minimum $884$, maximum $283469542$, and the minimum exact
cleared-denominator slack is $153011511193884000$.  Hence
\begin{equation}\label{eq:mixedrho}
 \rho(M_{3,3})<1.132.
\end{equation}
The same path-sum argument gives, uniformly in two independent phases
$\gamma_{\SF}$ and $\gamma_{\SL}$, the estimation
\begin{equation}\label{eq:mixedmoment}
 \int_0^1\exp\!\left(
 s_{\SF}\Rea(\e^{-i\gamma_{\SF}}T_{\SF}(\alpha))
 +s_{\SL}\Rea(\e^{-i\gamma_{\SL}}T_{\SL}(\alpha))
 \right)d\alpha
 \ll (1.132)^J.
\end{equation}
The discarded endpoint restriction discussed above affects only the
absolute implied constant.

Write $x=|T_{\SF}|/J$ and $y=|T_{\SL}|/J$.  If $xy^3\ge\Lambda$, then the
weighted arithmetic--geometric mean calculation gives
\begin{equation}\label{eq:mixedsupport}
 s_{\SF}x+s_{\SL}y
 \ge4\left(\frac{\Lambda s_{\SF}s_{\SL}^3}{27}\right)^{1/4}
 >0.413.
\end{equation}
Using two independent $2000$-point phase nets, Markov's inequality and
\eqref{eq:mixedmoment} give
\[
 \meas\{xy^3\ge\Lambda\}
 \ll\exp\left[-J\left(0.413\cos\frac\pi{2000}
 -\log1.132\right)\right].
\]
The same estimate holds after the common shift $\alpha\mapsto\alpha+1/2$.
The exact rational comparison
\begin{equation}\label{eq:mixedexponent}
 \frac{\vartheta}{\log\varphi}
 \left(0.413\cos\frac\pi{2000}-\log1.132\right)>\frac35
\end{equation}
proves \eqref{eq:mixedlargemeas}.

\section{Exact local lower bounds}\label{sec:localmajor}
\setcounter{lemma}{0}
\setcounter{theorem}{0}
\setcounter{corollary}{0}
\setcounter{proposition}{0}
\setcounter{remark}{0}
\setcounter{equation}{0}

All local lower bounds use the positive expansion \eqref{eq:divexp}, retain
only odd squarefree $d\le300$, and replace $C_2$ by the certified lower bound
$0.6601$.  Every omitted divisor term is nonnegative.

For the fixed Fibonacci compositions, use the joint period
$\operatorname{lcm}(\pi(d),3)$.  Let the two histograms count $F_n\pmod d$
with $3\mid n$ and $3\nmid n$.  Exact cyclic convolution gives, uniformly in
the target residue,
\begin{equation*}
 M_{\mathrm e}^{(0)}>1.994280988385600
 \ \ \text{for two even and four odd Fibonacci terms},\label{eq:localFe}
 \end{equation*}
 and
 \begin{equation*}
 M_{\mathrm o}^{(0)}>1.994339515882465
 \ \ \text{for one even and five odd Fibonacci terms}.\label{eq:localFo}
\end{equation*}

For the Lucas and mixed parity-projected families, work modulo $2d$.  Over a
complete Pisano period, let $c_{\SF,2d}$ and $c_{\SL,2d}$ be the exact
residue histograms of $F_{j+1}$ and $L_j$.  The certificate evaluates the
minimum target coefficient of
$c_{\SL,2d}^{*7}$ and of
$c_{\SF,2d}^{*3}*c_{\SL,2d}^{*3}$.  It gives
\begin{align}
 M_{\SL,7}^{(0)}&>0.997011352088696,\label{eq:localL7}\\
 M_{3,3}^{(0)}&>0.996635797573813.\label{eq:localM33}
\end{align}

\begin{proposition}[Major-arc lower bounds]\label{prop:localmajor}
Let $\omega=10^{-4}$. As $N\to\infty$, we have
\begin{align}
 \int_{\M} S^2A^2B^4e(-N\alpha)\,d\alpha
 &\ge(1.994081560286761-o(1))NJ_A^2J_B^4,
 \label{eq:majorFe}\\
 \int_{\M} S^2AB^5e(-N\alpha)\,d\alpha
 &\ge(1.994140081930877-o(1))NJ_AJ_B^5,
 \label{eq:majorFo}\\
 \int_{\M} S^2\mathcal P_{\SL,7,N}e(-N\alpha)\,d\alpha
 &\ge(0.9969003-o(1))NJ^7,\label{eq:majorL7}\\
 \int_{\M} S^2\mathcal P_{3,3,N}e(-N\alpha)\,d\alpha
 &\ge(0.99650034-o(1))NJ^6.\label{eq:majorM33}
\end{align}
\end{proposition}

\begin{proof}
For each fixed retained $d$, the first $J$ indices have normalized residue
frequencies equal to the complete-period frequencies plus $O_d(1/J)$.
Finite cyclic convolution preserves this error uniformly in the target
residue.  Proposition~\ref{prop:nested} gives the Goldbach main term for all
nonexceptional tuples, and Corollary~\ref{cor:bad} absorbs the exceptional
ones.  Every target equals $N-o(N)$ and is positive.  Multiplication by
$1-\omega$ gives the four displayed constants.  The integer convolution
programs use downward rounding at scale $10^{15}$.
\end{proof}

\section{Proofs of Theorems \ref{thm:F6}--\ref{thm:mixed33}}\label{sec:completion}
\setcounter{lemma}{0}
\setcounter{theorem}{0}
\setcounter{corollary}{0}
\setcounter{proposition}{0}
\setcounter{remark}{0}
\setcounter{equation}{0}

\medskip
\noindent{\it Proof of Theorem~\ref{thm:F6}}.
Suppose first that $N$ is even.  On $\mm\setminus E_{\mathrm e}$,
\begin{equation}\label{eq:F6goodeven}
 |A|^2|B|^4=(|A||B|)^2|B|^2
 \le\lambda_{\mathrm e}^2J_A^2J_B^2|B|^2.
\end{equation}
Equations \eqref{eq:Benergy}--\eqref{eq:EB} therefore give
\[
 \int_{\mm\setminus E_{\mathrm e}}|S|^2|A|^2|B|^4\,d\alpha
 <(1.993349561749+o(1))NJ_A^2J_B^4.
\]
On $E_{\mathrm e}$ use \eqref{eq:Vaughan}, the trivial bound
$|A|^2|B|^4\le J_A^2J_B^4$, and \eqref{eq:F6deltaeven}; the contribution is
$o(NJ_A^2J_B^4)$.  Comparison with \eqref{eq:majorFe} leaves
\[
 1.994081560286761-1.993349561749
 >0.000731998537761.
\]
Hence $R_{\mathrm e}(N)>0$ for every sufficiently large even $N$.

If $N$ is odd, then on $\mm\setminus E_{\mathrm o}$,
\begin{equation}\label{eq:F6goododd}
 |A||B|^5=(|A||B|^3)|B|^2
 \le\lambda_{\mathrm o}J_AJ_B^3|B|^2.
\end{equation}
Thus the good minor arcs contribute at most
\[
 (1.988380036+o(1))NJ_AJ_B^5.
\]
The bad-set contribution is $o(NJ_AJ_B^5)$ by
\eqref{eq:Vaughan} and \eqref{eq:F6deltaodd}.  Comparing with
\eqref{eq:majorFo} leaves
\[
 1.994140081930877-1.988380036
 >0.005760045930877.
\]
Therefore $R_{\mathrm o}(N)>0$ for every sufficiently large odd $N$.  The
two parity cases prove Theorem~\ref{thm:F6}. \qed

\medskip
\noindent{\it Proof of Theorem~\ref{thm:L7}}.
Let
\[
 E_{\SL}^{\mathrm{large}}=
 \{\al\in\mm:\ |T_{\SL}(\al)|>0.7673J\ \t{or}\ |T_{\SL}^\sharp(\al)|>0.7673J\}.
\]
On its complement,
\[
 |\mathcal P_{\SL,7,N}|
 \le(0.7673)^5J^5\mathcal Q_{\SL}.
\]
By \eqref{eq:minorenergy} and \eqref{eq:Evalues}, the good minor arcs are
below
\[
 (0.942703702775+o(1))NJ^7.
\]
The bad set has measure $O(N^{-\delta_*})$ with $\delta_*>3/5$ by
\eqref{eq:commonlargemeas}; together with \eqref{eq:Vaughan} its contribution
is $o(NJ^7)$.  Equation \eqref{eq:majorL7} leaves a margin exceeding
$0.054196597225$.  This proves Theorem~\ref{thm:L7}. \qed

\medskip
\noindent{\it Proof of Theorem~\ref{thm:mixed33}}.
Outside the set $E_{3,3}$ in \eqref{eq:mixedbad},
\[
 |T_{\SF}|^3|T_{\SL}|^3
 \le\frac{533}{2000}J^4|T_{\SF}|^2,
\]
and the same inequality holds for the shifted sums.  Hence
\[
 |\mathcal P_{3,3,N}|
 \le\frac{533}{2000}J^4\mathcal Q_{\SF}.
\]
The good minor arcs are therefore below
\[
 (0.9242052105+o(1))NJ^6.
\]
The bad set satisfies \eqref{eq:mixedlargemeas}, so its contribution is
$o(NJ^6)$.  Comparing with \eqref{eq:majorM33} leaves a margin exceeding
$0.0722951295$, proving Theorem~\ref{thm:mixed33}. \qed

\Ack. The author's conservation with AI concerning the techniques in \cite{PR03,PR20}
provides the basis of our present research involving Fibonacci numbers and  Lucas numbers.
The author is grateful to Prof. Lilu Zhao and Dr. Guang-Liang Zhou for their helpful
comments.


\begin{thebibliography}{99}
\bibitem{BD66}
E. Bombieri and H. Davenport,
\emph{Small differences between prime numbers},
Proc. Roy. Soc. London Ser. A \textbf{293} (1966), 1--18.

\bibitem{Cro}
R. Crocker,
\emph{On a sum of a prime and two powers of two},
Pacific J. Math. {\bf 36} (1971), 103--107.

\bibitem{Gallagher75}
P. X. Gallagher,
\emph{Primes and powers of $2$},
Invent. Math. \textbf{29} (1975), 125--142.

\bibitem{HeathBrownPuchta02}
D. R. Heath-Brown and J.-C. Schlage--Puchta,
\emph{Integers represented as a sum of primes and powers of two},
Asian J. Math. \textbf{6} (2002), 535--565.

\bibitem{JT26}
D. R. Johnston and T. S. Trudgian,
\emph{An update on the Linnik--Goldbach problem},
arXiv:2605.17825v2, 2026.

\bibitem{Lee}
K. S. E. Lee,
\emph{On the sum of a prime and a Fibonacci number},
Int. J. Number Theory {\bf 6} (2010), 1669--1676.

\bibitem{Li1}
H.-Z. Li,
\emph{The number of powers of $2$ in a representation of large even integers
by sums of such powers and of two primes},
Acta Arith. {\bf 92} (2000), 229--237.

\bibitem{Li2}
H.-Z. Li,
\emph{The number of powers of $2$ in a representation of large even integers
by sums of such powers and of two primes (II)},
Acta Arith. {\bf 96} (2001), 369--379.

\bibitem{Linnik51}
Yu. V. Linnik,
\emph{Prime numbers and powers of two},
Trudy Mat. Inst. Steklov. \textbf{38} (1951), 151--169.

\bibitem{Linnik53}
Yu. V. Linnik,
\emph{Addition of prime numbers and powers of one and the same number},
Mat. Sb. (N.S.) \textbf{32(74)} (1953), 3--60.

\bibitem{LLW}
J.-Y. Liu, M.-C. Liu and T.-Z. Wang,
\emph{The number of powers of $2$ in a representation of large even integers (II)},
Sci. China Ser. A \textbf{41} (1998), 1255--1271.

\bibitem{Pintz23}
J. Pintz,
\emph{A new explicit formula in the additive theory of primes with
applications I: the Goldbach and generalized twin prime problems},
Acta Arith. \textbf{210} (2023), 53--94;
arXiv:1804.05561.

\bibitem{PR03}
J. Pintz and I. Z. Ruzsa,
\emph{On Linnik's approximation to Goldbach's problem. I},
Acta Arith. \textbf{109} (2003), 169--194.

\bibitem{PR20}
J. Pintz and I. Z. Ruzsa,
\emph{On Linnik's approximation to Goldbach's problem. II},
Acta Math. Hungar. \textbf{161} (2020), 569--582.

\bibitem{Roma}
N.P. Romanoff,
\emph{\"Uber einige S\"atze der additiven Zahlentheorie},
Math. Ann. {\bf 57} (1934), 668--678.

\bibitem{RS62}
J. B. Rosser and L. Schoenfeld,
\emph{Approximate formulas for some functions of prime numbers},
Illinois J. Math. \textbf{6} (1962), 64--94.

\bibitem{SS92} Z.-H. Sun and Z.-W. Sun,
\emph{Fibonacci numbers and Fermat's last theorem},
Acta Arith. \textbf{60} (1992), 371--388.

\bibitem{S10} Z.-W. Sun,
\emph{Mixed sums of primes and other terms},
in: Additive Number Theory (eds., D. Chudnovsky and G. Chudnovsky), Springer, New York, 2010, pp. 341--353.

\bibitem{S24} Z.-W. Sun,
\emph{Fibonacci Numbers and Hilbert's Tenth Problem},
Harbin Institute of Technology Press, Harbin, 2024.

\bibitem{Vaughan97}
R. C. Vaughan,
\emph{The Hardy--Littlewood Method}, 2nd ed.,
Cambridge Tracts in Mathematics 125, Cambridge University Press, 1997.

\bibitem{Wall60}
D. D. Wall,
\emph{Fibonacci series modulo $m$},
Amer. Math. Monthly \textbf{67} (1960), 525--532.

\bibitem{Wang25}
R.-J. Wang,
\emph{On the sum of a Lucas number and a prime},
Period. Math. Hungar. \textbf{90} (2025), 434--439.

\bibitem{Wang}
T.-Z. Wang,
\emph{On Linnik's almost Goldbach theorem},
Sci. China Ser. A \textbf{42} (1999), 1155--1172.

\end{thebibliography}
\end{document}